\documentclass[english]{article}
\usepackage[T1]{fontenc}
\usepackage[latin9]{inputenc}
\usepackage{color}
\usepackage{amssymb}

\usepackage[T1]{fontenc}
\usepackage[latin9]{inputenc}
\usepackage{color}
\usepackage{amsmath}
\usepackage{graphicx}
\usepackage{amssymb}
\usepackage{amsfonts}

\usepackage{babel}%
\providecommand{\U}[1]{\protect\rule{.1in}{.1in}}

\newtheorem{theorem}{Theorem}

\newenvironment{proof}[1][Proof]{\noindent\textbf{#1.} }{\ \rule{0.5em}{0.5em}}

\usepackage{babel}

\begin{document}

\title{A Sum Theorem for (FPV) Operators and Normal Cones}

\author{M.D. Voisei}
\date{}

\maketitle
\begin{abstract}
On \cite[p.\ 199]{MR2386931} one says ``We mention parenthetically
that the proof of {[}99, Lemma 41.3{]} is incorrect, and we do not
know whether it, {[}99, Theorem 41.5{]} and {[}99, Theorem 41.6{]}
are true''. The previously cited reference {[}99{]} is our reference
\cite{MR1723737}. The aim of this short note is to provide a result
that improves upon \cite[Lemma 41.\ 3]{MR1723737}.
\end{abstract}

Recall that in the context of a Banach space $X$ with dual $X^{*}$
and coupling $c(x,x^{*})=\langle x,x^{*}\rangle=x^{*}(x)$, $(x,x^{*})\in X\times X^{*}$:
\begin{itemize}
\item $\varphi_{S}(x,x^{*})=\sup\{\langle x-s,s^{*}\rangle+\langle s,x^{*}\rangle\mid(s,s^{*})\}$,
$(x,x^{*})\in X\times X^{*}$ stands for the Fitzpatrick function
of $S\subset X\times X^{*}$,
\item $z=(x,x^{*})$ is monotonically related to (m.r.t. for short) $S$
comes to $z\in[\varphi_{S}\le c]:=\{w\in X\times X^{*}\mid\varphi_{S}(w)\le c(w)\}$,
\item $A$ is of type {\it (FPV)} if for every open convex $V\subset X$ with
$V\cap D(A)\neq\emptyset$ if $z=(x,x^{*})$ is monotonically related
to (m.r.t. for short) $A|_{V}$ and $x\in V$ then $z\in A$ or equivalently
if $z=(x,x^{*})\not\in A$ and $x\in V$ then there is $(a,a^{*})\in A|_{V}$
such that $\langle x-a,x^{*}-a^{*}\rangle<0$. Here $\operatorname*{Graph}(A|_{S})=\operatorname*{Graph}(A)\cap S\times X^{*}$,
$S\subset X$ (see e.g. \cite[p.\ 268]{MR1249266}, \cite[Def.\ 36.7]{MR2386931}).
In other words $A$ is of type (FPV) if, for every open convex $V\subset X$
with $V\cap D(A)\neq\emptyset$, $A|_{V}$ is maximal monotone in
$V\times X^{*}$,
\item $x\in\operatorname*{cen}D(A)$ means the segment $[x,y]:=\{tx+(1-t)y\mid0\le t\le1\}\subset D(A)$,
for every $y\in D(A)$.
\end{itemize}

Let us introduce a new class of operators:
\begin{itemize}
\item $A$ is called of type \emph{weak-(FPV)} if for every open convex
$V\subset X$ with $V\cap D(A)\neq\emptyset$ if $z=(x,x^{*})$ is
monotonically related to $A|_{V}$ and $x\in V$ then $x\in D(A)$
or equivalently for every $z=(x,x^{*})\in V\setminus D(A)\times X^{*}$
there is $(a,a^{*})\in A|_{V}$ such that $\langle x-a,x^{*}-a^{*}\rangle<0$.
In other words $A$ is of type weak-(FPV) if for every open convex
$V\subset X$ with $V\cap D(A)\neq\emptyset$, $A|_{V}$ cannot be
extended, as a monotone operator in $V\times X^{*}$, outside $D(A)\cap V$.
\end{itemize}

\begin{theorem} \label{T1}Let $X$ be a Banach space, $C\subset X$
closed convex, and $A:X\rightrightarrows X^{*}$ be \textcolor{black}{maximal
monotone and of type weak-(FPV) with} $\operatorname*{cen}D(A)\cap\operatorname*{int}C\neq\emptyset$.
Then $A+N_{C}$ is maximal monotone. \end{theorem}

\begin{proof} Without loss of generality we may assume that $0\in\operatorname*{cen}D(A)\cap\operatorname*{int}C$,
$0\in A0$, and for some $r>0$, $rU\subset\operatorname*{int}C$,
where $U$ denotes the unit open ball in $X$. Since $A+N_{C}$ is
representable (see \cite[Cor.\ 5.6]{tscr}) it remains to prove that
$A+N_{C}$ is NI (see \cite[Remark.\ 3.5]{tscr}). Assume by contradiction
that $A+N_{C}$ is not NI, that is, there is $z=(x,x^{*})\in[\varphi_{A+N_{C}}<c]:=\{w\in X\times X^{*}\mid\varphi_{A+N_{C}}(w)<c(w)\}$.
Since for every $y\in C$, $N_{C}(y)$ is a cone, note that

$\bar{z}=(\bar{x},\bar{x}^{*})\ {\rm is\ m.r.t.}\ A+N_{C}\Leftrightarrow$

\vspace{-.3cm}

\begin{equation}
\bar{z}\ {\rm is\ m.r.t.}\ A|_{C}\ {\rm and}\ \langle\bar{x}-a,x^{*}\rangle\le0,\ a\in D(A)\cap C,\ x^{*}\in N_{C}(a).\label{mrts}\end{equation}

Therefore $z$ is m.r.t. $A|_{C}$ and \begin{equation}
\langle x-a,x^{*}\rangle\le0,\ a\in D(A)\cap C,\ x^{*}\in N_{C}(a).\label{mrt0}\end{equation}

Assume that $x\in D(A)$. Since $z\in[\varphi_{A+N_{C}}<c]$
we know that $x\not\in C$ (see \cite[Prop.\ 2.1\ (d)]{tscr}). Therefore,
there is $\mu\in(0,1)$ such that $\mu x\in D(A)\cap\operatorname*{Fr}C$
(recall that $0\in\operatorname*{cen}D(A)\cap\operatorname*{int}C$).
Take $x^{*}\in N_{C}(\mu x)$, such that $\langle\mu x-y,x^{*}\rangle>0$,
for every $y\in\operatorname*{int}C$; whence $\langle x,x^{*}\rangle>0$,
because $0\in\operatorname*{int}C$ and $\mu>0$. From
(\ref{mrt0}) applied for $a=\mu x$ and since $\mu<1$ one gets the
contradiction $\langle x,x^{*}\rangle\le0$.

Therefore $x\not\in D(A)$. For $n\ge1$, let $V_{n}:=[0,x]+\frac{1}{n}U$.
Notice that $V_{n}$ is open convex, $V_{n}\cap D(A)\neq\emptyset$,
and $x\in V_{n}$, $n\ge1$. Since $A$ is weak-(FPV), for every $n\ge1$,
there is $z_{n}=(a_{n},a_{n}^{*})\in A$ such that $a_{n}\in V_{n}$
and $c(z-z_{n})<0$. This implies that $a_{n}\in D(A)\setminus C$,
because $z$ is m.r.t. $A|_{C}$. Hence there is $t_{n}\in(0,1)$
such that $x_{n}=t_{n}a_{n}\in\operatorname*{Fr}C\cap D(A)$, since
$0\in\operatorname*{cen}D(A)\cap\operatorname*{int}C$. Let $x_{n}^{*}\in N_{C}(x_{n})$,
$\|x_{n}^{*}\|=1$, $n\ge1$. Because $x_{n}\in V_{n}$ there is $\lambda_{n}\in[0,1]$
such that $\|x_{n}-\lambda_{n}x\|\le\frac{1}{n}$, $n\ge1$. On a
subnet, denoted by the same index for simplicity, we may assume that
$\lambda_{n}\rightarrow\lambda\in[0,1]$, $x_{n}\rightarrow\lambda x\in\operatorname*{Fr}C$,
strongly in $X,$ $x_{n}^{*}\rightarrow x^{*}\in N_{C}(\lambda x)$,
weakly-star in $X^{*}$ as $n\rightarrow\infty$. Note that $\lambda>0$
because $\lambda x\in\operatorname*{Fr}C$ and $0\in\operatorname*{int}C$.

By the monotonicity of $N_{C}$ for $0\in N_{C}(ru)$, $\|u\|<1$,
we get $\langle x_{n}-ru,x_{n}^{*}\rangle\ge0$ or $\langle x_{n},x_{n}^{*}\rangle\ge r$,
$n\ge1$. Let $n\rightarrow\infty$ to find $\langle x,x^{*}\rangle\ge r/\lambda>0$.
From (\ref{mrt0}) we have \textcolor{black}{$\langle x-x_{n},x_{n}^{*}\rangle\le0,$
and after we pass to limit, we get $(1-\lambda)\langle x,x^{*}\rangle\le0$,
$\lambda=1,$ and so $x\in\operatorname*{Fr}C$.}

Consider $f(t)=(\varphi_{A+N_{C}}-c)(tz)$, $t\in\mathbb{R}$;
$f$ is continuous on its domain (an interval) with $f(0)=0$ and
$f(1)<0$. Therefore there is $0<t<1$ such that $f(t)<0$. This implies
that $tz$ is m.r.t $A+N_{C}$ (in particular, according to (\ref{mrts}),
$tz$ is m.r.t. $A|_{C}$) with $tx\in\operatorname*{int}C$, so $tx\in D(A)$,
since $A$ is weak-(FPV). From $tx\in D(A)\cap C$ we get the contradiction
$f(tz)\ge0$, that is, $\varphi_{A+N_{C}}(tz)\ge c(tz)$ (see again
\cite[Prop.\ 2.1\ (d)]{tscr}).

This contradiction occurred due to the consideration
of the assumption that $A+N_{C}$ is not NI. Hence $A+N_{C}$ is NI
and consequently maximal monotone (see \cite[Th.\ 3.4]{tscr}). \end{proof}

\strut

Therefore \cite[Lemma\ 41.3]{MR1723737} and its consequence \cite[Th.\ 41.5]{MR1723737}
are true. We mention also that a multi-valued version of \cite[Th.\ 41.6]{MR1723737}
has been proved in \cite{Bau-arxiv}.

\end{document}